\documentclass{amsart}
\usepackage{amsmath,amsthm,amscd,amssymb}
\usepackage{latexsym}
\usepackage{txfonts}

\numberwithin{equation}{section}

\theoremstyle{plain}
\newtheorem{theorem}[equation]{Theorem}
\newtheorem{lemma}[equation]{Lemma}

\theoremstyle{definition}

\theoremstyle{remark}

\newcommand{\dv}{\operatorname{div}}

\begin{document}

\title{Dahlberg's bilinear estimate for solutions of divergence form 
complex elliptic equations}

\author[S. Hofmann]{Steve Hofmann}
\address{Department of Mathematics, 
University of Missouri, Columbia, Missouri 65211, USA}
\email{hofmann@math.missouri.edu}
\thanks{The author was supported by the National Science Foundation}


\maketitle
\begin{abstract} We consider divergence form elliptic operators $L=-\dv A(x)\nabla$, defined in
$\mathbb{R}^{n+1}=\{(x,t)\in\mathbb{R}^{n}\times\mathbb{R}\},\, n \geq 2$, where the $L^{\infty}$ coefficient matrix $A$ is
$(n+1)\times(n+1)$, uniformly elliptic, complex and $t$-independent. Using recently obtained results
concerning the boundedness and invertibility of layer potentials associated to such operators, we show that if $Lu=0$ in $\mathbb{R}^{n+1}_+$, then for any vector-valued ${\bf v} \in W^{1,2}_{loc},$ we have the bilinear estimate
$$\left|\iint_{\mathbb{R}^{n+1}_+} \nabla u \cdot \overline{{\bf v}} dx dt \right|\leq C\sup_{t>0}
\|u(\cdot,t)\|_{L^2(\mathbb{R}^n)}\left( \||t \nabla {\bf v}\|| +
\|N_*{\bf v}\|_{L^2(\mathbb{R}^n)}\right),$$
where $\||F\|| \equiv \left(\iint_{\mathbb{R}^{n+1}_+} |F(x,t)|^2 t^{-1} dx dt\right)^{1/2},$ and where $N_*$ is the usual non-tangential maximal operator.  The result is new even in the case
of real symmetric coefficients, and generalizes the analogous result of Dahlberg for harmonic functions on Lipschitz graph domains.
\end{abstract}

\section{Introduction \label{s1}}

In \cite{D}, B. Dahlberg considered the bilinear singular integral form
\begin{equation} \int_\Omega \nabla u \cdot\overline{{\bf v}},\end{equation} where 
$u$ is harmonic in the domain
$\Omega \equiv \{(x,t) \in \mathbb{R}^{n+1}: t > \varphi (x)\},$ with $\varphi $ Lipschitz, and
where ${\bf v} \in W^{1,2}_{loc}$ is vector valued.  He showed that the bilinear form (1.1) is bounded by
the $L^2$ norm of the square function plus the non-tangential maximal function of $u$, times the same expression 
for ${\bf v}$.  In the present note, we generalize Dahlberg's Theorem to variable coefficient divergence form elliptic operators.  To be precise, let 
\begin{equation*} L=-\dv A\nabla\equiv-\sum_{i,j=1}^{n+1}\frac{\partial}{\partial
x_{i}}\left(A_{i,j} \,\frac{\partial}{\partial x_{j}}\right)\end{equation*}
be defined in $\mathbb{R}^{n+1}=\{(x,t)\in\mathbb{R}^{n}\times\mathbb{R}\},\, n\geq 2,$ (we use the convention that
$x_{n+1}=t$), where $A=A(x)$ is an $(n+1)\times(n+1)$ matrix of complex-valued $L^{\infty}$ coefficients, defined on
$\mathbb{R}^{n}$ (i.e., independent of the $t$ variable), and satisfying the 
uniform ellipticity (accretivity) condition
\begin{equation}
\label{eq1.1} \lambda|\xi|^{2}\leq\Re e\,\langle A(x)\xi,\xi\rangle, \, \,\,
  \Vert A\Vert_{L^{\infty}(\mathbb{R}^{n})}\leq\Lambda,
\end{equation}
 for some $\lambda>0$, $\Lambda<\infty$, and for all $\xi\in\mathbb{C}^{n+1}$, $x\in\mathbb{R}^{n}$. Here,
$\langle\cdot,\cdot\rangle$ denotes the usual hermitian inner product in $\mathbb{C}^{n+1}$, so that \begin{equation*}
\langle A(x)\xi,\xi\rangle\equiv\sum_{i,j=1}^{n+1}A_{ij}(x)\xi_{j}\bar{\xi_{i}}\end{equation*}

In order to state our theorem,
we first recall that the non-tangential maximal operator $N_{\ast}$ (and a variant
$\widetilde{N}_*$)
are defined as follows. Given $x_0\in\mathbb{R}^{n}$, let
$$\gamma(x_0)=\{(x,t)\in\mathbb{R}_{+}^{n+1}:|x_0-x|<t\}$$
denote the cone with vertex at $x_0$. Then
for $U $ defined in $\mathbb{R}_{+}^{n+1}$,
\begin{equation*} N_{\ast} U(x_0)  \equiv\sup_{(x,t)\in\gamma(x_0)}|U(x,t)|,\quad
\widetilde{N}_{\ast} U(x_0)  \equiv\sup_{(x,t)\in \gamma(x_0)}
\left(\fint\!\!\fint_{\substack{|x-y|<t\\ |t-s|<t/2}}|U(y,s)|^{2}dyds\right)^{\frac{1}{2}}.\end{equation*}

Our main result is the following:
\begin{theorem}  \label{t1.3} Suppose that $L$ is an operator of the type described above, with 
\begin{equation}\label{eq1.small}\|A - A_0\|_\infty \leq \epsilon ,\end{equation} for some real, symmetric, $L^\infty$, elliptic, and $t$-independent matrix  $A_0$.  Suppose also that
$Lu=0$, and that ${\bf v} \in W^{1,2}_{loc}(\mathbb{R}^{n+1}, \mathbb{C}^{n+1}).$  
If $\epsilon\leq\epsilon_0$, with $\epsilon_0$ sufficiently small, depending only on dimension and ellipticity, 
then we have the bilinear estimate
$$\left|\iint_{\mathbb{R}^{n+1}_+} \nabla u \cdot \overline{{\bf v}} \, dx dt \,\right| \leq C
\sup_{t>0}
\|u(\cdot,t)\|_{L^2(\mathbb{R}^n)}\left( \||t \nabla {\bf v}\|| +
\|N_*{\bf v}\|_{L^2(\mathbb{R}^n)}\right),$$
where $ C = C(n,\lambda,\Lambda)$ and $$\||F\|| \equiv \left(\iint_{\mathbb{R}^{n+1}_+} |F(x,t)|^2 t^{-1} dx dt\right)^{1/2}.$$
\end{theorem}

We remark that the result is new even in the case of real, symmetric coefficients.  
The analogous result was proved by Dahlberg
for harmonic functions in Lipschitz graph domains, using a special change of variable found by 
Kenig and Stein,
and independently by Maz'ya.  Our theorem includes that of Dahlberg, as may be seen by pulling back under the mapping $(x,t) \to (x,\varphi (x) +t)$.  Dahlberg's original method seems inapplicable 
to the variable coefficient case, unless the
coefficients are differentiable and satisfy an appropriate sort of Carleson condition
as in the work of Kenig and Pipher \cite{KP}.  In the present setting, in lieu of the special change of variable, we use recently obtained results of \cite {AAAHK} concerning the boundedness and invertibility of layer potentials associated to variable coefficient $t$-independent operators.

The paper is organized as follows.  In the next section, we recall some  of the aforementioned results 
of \cite{AAAHK}.  In Section \ref{s3}, we prove Theorem \ref{t1.3},
and in Section \ref{s4} we discuss the analogue of another result of \cite{D} 
concerning the domain of the infinitesimal generator of the Poisson semigroup for the equation $Lu=0$
in $\mathbb{R}_+^{n+1}$.

Let us now set some notation  and terminology that we shall use in the 
sequel.  We shall employ the standard convention that the generic constant $C$
is allowed to vary from one instance to the next, and may depend upon dimension
and ellipticity.  The symbol $\fint$ denotes the mean value, i.e.,
$\fint_E f \equiv |E|^{-1} \int_E f .$  We shall use the notation
$$D_j\equiv\partial_{x_j}, \,\,\,1\leq j\leq n+1,$$
bearing in mind that $x_{n+1} = t,$ and we use $e_j,\, 1\leq j\leq n+1,$ to indicate the standard
unit basis vector in the $x_j$ direction. The symbol $\nabla$ denotes the full $(n+1)$-dimensional 
gradient, acting in both $x$ and $t$, and we use $\nabla_x$ or $\nabla_\|$ to indicate the $n$-dimensional gradient acting only in $x$.  We use $ad\!j$ to denote the hermitian
adjoint of an operator acting on functions defined on $\mathbb{R}^n$.
We define the homogeneous Sobolev space $\dot{L}_{1}^{2}$ to be the 
completion of $C_0^{\infty}$ with respect to the seminorm
$\|\nabla F\|_2.$  As is well known, for $n\geq 2$, this space can be identified (modulo constants) with the space $I_1(L^2) \equiv \Delta^{-1/2}(L^2)$. 
We write $F \to f\,n.t.$ to mean that for $a.e.$ 
 $x \in \mathbb{R}^n$, $F(y,t) \to f(x)$, as $(y,t) \to (x,0)$, with
 $(y,t) \in \gamma(x).$

\section{Results for variable coefficient layer potentials \label{s2}}

We now recall the definitions of the layer potentials.  
We first note that by \eqref{eq1.small}, the stability result of \cite{A}, 
and the classical De Giorgi-Nash Theorem 
\cite{DeG,N}, solutions of $Lu=0$ are locally H\"older continuous.
Let $\Gamma (x,t,y,s)$ denote 
the fundamental solution for $L$ (we refer the reader to
\cite{HK} for the construction of, and estimates for, $\Gamma$ in the case of complex coefficients,
assuming De Giorgi-Nash bounds).
By $t$-independence,
\begin{equation}\label{2.1a}\Gamma (x,t,y,s) \equiv \Gamma (x,t-s,y,0).\end{equation}
We define the single and double layer potentials, respectively, in the usual way:

 \begin{equation}
\begin{split}\label{eq1.5}S_{t}f(x) & \equiv\int_{\mathbb{R}^{n}}\Gamma(x,t,y,0)\,f(y)\,dy, \,\,\, t\in \mathbb{R}\\
\mathcal{D}_{t}f(x) & 
\equiv\int_{\mathbb{R}^{n}}\overline{\partial_{\nu^*(y)} \Gamma^*
(y,0,x,t)}\,f(y)\,dy,\,\,\, t \neq 0,\end{split}
\end{equation}
 where $\partial_{\nu^*}$ is the adjoint exterior
 conormal derivative; i.e., if $A^{\ast}$ denotes the hermitian adjoint of $A$, then
\begin{equation}\label{eq2.2a}
\partial_{\nu^*(y)} \Gamma^* (y,0,x,t)
=-\sum^{n+1}_{j=1}A_{n+1,j}^{\ast}(y)\frac{\partial \Gamma^*}{\partial
y_{j}}(y,0,x,t)=-e_{n+1}\cdot A^{\ast}(y)
\nabla_{y,s}\Gamma^*(y,s,x,t) \mid_{s=0}\end{equation}
 (recall that $y_{n+1}=s$), where $e_{n+1}\equiv (0,...0,1)$ is the unit basis vector in the $t$ direction.
 Here, $\Gamma^*$ is the fundamental solution for
 $L^*$, the hermitian adjoint of $L$.  Thus, $\Gamma^*$ is the conjugate transpose of $\Gamma$; i.e.,
 $$\overline{\Gamma^*(y,s,x,t)} =\Gamma(x,t,y,s).$$   We also define (formally) the
 boundary singular integral
\begin{equation}
\label{eq1.6}Kf(x) \equiv
  ``p.v."\int_{\mathbb{R}^{n}}\overline{\partial_{\nu^*} \Gamma^* (y,0,x,0)}\,f(y)\,dy.
\end{equation}
 (For  the precise definiton of the latter operator in the 
 case of non-smooth coefficients,   see \cite{AAAHK}, Section 4).
In a departure from tradition impelled by the context of complex coefficients,
 $K^*, S^*$ and $\mathcal{D}^*$ will denote the analogues of $K,S$ and $\mathcal{D}$ corresponding to $L^*$.  

In order to prove our Theorem, we shall require some of
 the main results of \cite{AAAHK}, which we summarize as follows:
\begin{theorem} \cite{AAAHK}.   \label{t2.6} Suppose that $L$ satisfies the hypotheses of Theorem 1.3.  There exists a small constant $\epsilon_0= \epsilon_0(n,\lambda,\Lambda)$ such that if $\epsilon$ in \eqref{eq1.small} satisfies $\epsilon \leq \epsilon_0$,
then the layer potential operators $\pm \frac{1}{2} I
+ K,\, \pm\frac{1}{2}I + K^*$ are isomorphisms on $L^2 (\mathbb{R}^n)$,
with the implicit constants depending only upon
dimension and ellipticity.  
In addition,
\begin{equation}\label{eq2.squarefunction}\sup_{t>0}\|\mathcal{D}_t f\|_2+
\||t \nabla \partial_t S_{\pm t} f\|| +\sup_{1\leq j \leq n}
\||t  \partial_t S_{\pm t} D_jf\||\leq C \|f\|_2,\end{equation}
and $\mathcal{D}_{\pm t}f \to (\pm\frac{1}{2}I + K)f$ $n.t.$ and in $L^2$, for $f\in L^2$.
Moreover, the corresponding statements hold also for $\mathcal{D}_t^*, K^*$ and $S_t^*$.  Finally, 
the solution to the Dirichlet problem
\begin{equation}
\begin{cases} Lu=0\text{ in }\mathbb{R}_{+}^{n+1}=\{(x,t)\in\mathbb{R}^{n}\times(0,\infty)\}\\ 
\lim_{t\to 0}u(\cdot,t)=f\text{ in }
L^{2}(\mathbb{R}^{n}) \text{ and } n.t.\\ 
\sup_{t>0}\|u(\cdot,t)\|_{L^2(\mathbb{R}^n)}<\infty,
\end{cases}\tag{D2}\label{D2}\end{equation}
which exists by virtue of the aforementioned facts about layer potentials, is unique.
\end{theorem}

We shall also require the following technical facts.
\begin{lemma}\label{l2.9a} \cite{AAAHK} (Lemma 2.2)
Suppose that $L$ satisfies the hypotheses of Theorem \ref{t1.3},
with $\epsilon \leq \epsilon_0$.  Set $K_t(x,y) =t\partial_t^2\,\Gamma(x,t,y,0)$.
Then \begin{equation}\label{eq2.10a}
|K_t(x,y)|\leq C \frac{t}{(t + |x-y|)^{n+1}}.
\end{equation}
\end{lemma}
\begin{lemma}\label{l2.11a}\cite{AAAHK} (Lemma 2.8)
Suppose that $L$ satisfies the hypotheses of Theorem \ref{t1.3},
with $\epsilon \leq \epsilon_0$. Then
$$\|t^2\nabla\partial_t^2S_{\pm t}^* f\|_{L^2(\mathbb{R}^n)} \leq C\|f\|_2.$$ 
\end{lemma}
\begin{lemma}\label{l2.15}. 
Suppose that $\{ R_t\}_{t>0}$ is a 
family of operators defined by $$R_tf(x)\equiv \int_{\mathbb{R}^n}K_t(x,y) f(y) dy, $$
where the kernel $K_t$ satisfies  \eqref{eq2.10a}.  Suppose also that $R_t1=0$ for all $t\in \mathbb{R}$. Then for $h\in \dot{L}^2_1 (\mathbb{R}^n)$,
\begin{equation}\label{eq2.15a}\int_{\mathbb{R}^n}|R_th|^2\leq Ct^2\int_{\mathbb{R}^n}|\nabla_x h|^2.\end{equation}
\end{lemma}

The proof of the last lemma is a standard exercise in the use of Poincar\'e's inequality.  We omit the details, but see, e.g. \cite{AAAHK} (Lemma 3.5), for a more general result.

Finally, we shall use the following special case of the ``Fatou Theorem" of \cite{AAAHK}.
\begin{lemma}\label{l2.12}  \cite{AAAHK} (Corollary 4.41) Suppose that $L$ satisfies the hypotheses of Theorem
\ref{t1.3}, with $\epsilon$ sufficiently small, and that $Lu=0$ in $\mathbb{R}^{n+1}.$   
Suppose also that
\begin{equation*} \sup_{t>0} \|u(\cdot,t)\|_2 <\infty.
\end{equation*}
Then $u(\cdot,t)$ converges $n.t.$ and in $L^2$ as $t \to 0$.
\end{lemma}

\section{Proof of Theorem \ref{t1.3} \label{s3}}
The proof of the theorem will use the following
\begin{lemma}\label{l3.0}Suppose that $L$ satisfies the hypotheses of Theorem \ref{t1.3},
with $\epsilon$ sufficiently small.   Suppose also that $Lu=0$ in $\mathbb{R}^{n+1}$,
with $\sup_{t>0}\|u(\cdot,t)\|_2 < \infty.$  Then 
$$\||t\nabla u\|| \leq C \sup_{t>0}\|u(\cdot,t)\|_2.$$
\end{lemma}
\begin{proof}
It is enough to show that for each fixed $\eta \in (0,10^{-10}), $
$$ \fint_\eta^{2\eta}\int_\delta^{1/\delta}\!\!\int_{\mathbb{R}^n}
|\nabla u|^2 t dx dt d\delta\leq C\sup_{t>0}\|u(\cdot,t)\|_2.$$ 
Integrating by parts in $t$ on the left side of the last inequality, we obtain 
\begin{equation}\label{eq3.0b}-\Re e \fint_\eta^{2\eta}\int_\delta^{1/\delta}\langle
\partial_t \nabla u,\nabla u\rangle t^2 dt d\delta + \text { boundary},\end{equation}
where the boundary terms are dominated by
\begin{equation}\label{eq3.0a}\sup_{r>0} \fint_r^{2r}\!\!\int_{\mathbb{R}^n}r^2|\nabla u(x,t)|^2dxdt \leq C\sup_{t>0}\|u(\cdot,t)\|^2_2,\end{equation}
as desired, and in the last step we have split $\mathbb{R}^n$ into cubes of side length $\approx r$ and used Caccioppoli's inequality.
By Cauchy's inequality, the main term in \eqref{eq3.0b} is no larger than 
$$\frac{\varepsilon}{2} \fint_\eta^{2\eta}\int_\delta^{1/\delta}\!\!\int_{\mathbb{R}^n}
|\nabla u|^2 t dx dt d\delta +  \frac{1}{2\varepsilon}\fint_\eta^{2\eta}
\int_\delta^{1/\delta}\!\!\int_{\mathbb{R}^n}
|\nabla \partial_t u|^2 t^3 dx dt d\delta \equiv I + II,$$
where $\varepsilon >0$ is at our disposal.  Choosing $\varepsilon$ small, we may hide term $I$.  Having fixed  $\varepsilon$, and applying Caccioppoli's inequality in Whitney boxes, we obtain that
$$II \leq C \iint_{\mathbb{R}^{n+1}_+} |\partial_t u|^2 dx \,t dt.$$
By the Fatou Theorem of \cite{AAAHK}, Section 4, $u$ converges in $L^2(\mathbb{R}^n)$ to some
$f $, with $$\|f\|_2 \leq \sup_{t>0}\|u(\cdot,t)\|_2.$$ Thus,  by Theorem \ref{t2.6}, 
$u(\cdot,t) = \mathcal{D}_t \left(-\frac{1}{2}I + K\right)^{-1}f.$
By \eqref{eq2.squarefunction}, the bijectivity of $\left(-\frac{1}{2}I+K\right)$ 
and the definition of $\mathcal{D}_t$, we obtain that
$$\iint_{\mathbb{R}^{n+1}_+} |\partial_t u|^2 dx \,t dt \leq C\|f\|_2\leq C\sup_{t>0}\|u(\cdot,t)\|_2.$$
\end{proof}

\begin{proof}[Proof of Theorem \ref{t1.3}]By the previous lemma, it is enough to establish the bound
\begin{multline}\label{eq3.1}
\sup_{0<\eta<10^{-10}}\fint_\eta^{2\eta}
\left|\int_\delta^{1/\delta}\!\!\int_{\mathbb{R}^n} \nabla u\cdot \overline{{\bf v}}\, dx dt \right|d\delta\\
\leq C
\left(\||t \nabla u\||+\sup_{t>0}
\|u(\cdot,t)\|_{L^2(\mathbb{R}^n)}\right)\left( \||t \nabla {\bf v}\|| +\sup_{t>0}
\|{\bf v}(\cdot,t)\|_{L^2(\mathbb{R}^n)}\right).
\end{multline}
We may suppose that the right hand side of
\eqref{eq3.1} is finite, otherwise there is nothing to prove.
On the left hand side of \eqref{eq3.1}, for each fixed $\eta$, we integrate by parts in $t$
to obtain the bound
\begin{equation}\label{eq3.2}\left|\fint_\eta^{2\eta}\! \!\int_\delta^{1/\delta}\! \!\int_{\mathbb{R}^n} \nabla 
u\cdot \overline{\partial_t{\bf v}} \,dx\, tdt d\delta\right|
+\left|\fint_\eta^{2\eta}\! \!\int_\delta^{1/\delta}\! \!\int_{\mathbb{R}^n} \nabla
\partial_t u\cdot  \overline{{\bf v}} \,dx\,t dtd\delta \right|
+\text{ boundary},\end{equation}
where the boundary terms are dominated by
\begin{multline*}C\left(\sup_{r>0} \fint_r^{2r}\!\!\int_{\mathbb{R}^n}r^2|\nabla u(x,t)|^2dxdt
\right)^{1/2}\left( \sup_{t>0}\|{\bf v}(\cdot,t)\|_2\right)\\\leq\,\, C\left( \sup_{t>0}\|u(\cdot,t)\|_2\right)
\left( \sup_{t>0}\|{\bf v}(\cdot,t)\|_2\right),\end{multline*}
and we have used \eqref{eq3.0a} in the last step.   Moreover, by Cauchy-Schwarz, the
first term in \eqref{eq3.2} is no larger than
$\||t\nabla u\|| \,\, \||t\nabla {\bf v}\||.$

It therefore remains to treat the middle term in \eqref{eq3.2}.  To this end, we write
$\nabla = \nabla_x + \partial_t e_{n+1},$ and ${\bf v} = \text{v}_\| + \text{v}_{n+1}e_{n+1},$
where $v_{n+1}\equiv {\bf v}\cdot e_{n+1}.$
Now, \begin{multline*}\left|\fint_\eta^{2\eta}\! \!\int_\delta^{1/\delta}\! \!\int_{\mathbb{R}^n} \nabla_x
\partial_t u\cdot \overline{\text{v}_\|} \,dx\,t dtd\delta \right| \\=\,\,
\left|\fint_\eta^{2\eta}\! \!\int_\delta^{1/\delta}\! \!\int_{\mathbb{R}^n} 
\partial_t u\,\dv_x \!\overline{\text{v}_\|}  \,dx\,t dtd\delta \right|
\,\,\leq\,\, C\||t\nabla u\|| \,\, \||t\nabla {\bf v}\||,\end{multline*}
as desired.  Thus, it is enough to consider
\begin{multline}\label{eq3.3}\left|\fint_\eta^{2\eta}\! \!\int_\delta^{1/\delta}\! \!\int_{\mathbb{R}^n} 
\partial_t^2 u\, \overline{\text{v}_{n+1}}dx\,t dtd\delta \right|\,\,\leq\,\,
\frac{1}{2}\left|\fint_\eta^{2\eta}\! \!\int_\delta^{1/\delta}\! \!\int_{\mathbb{R}^n} 
\partial_t^3 u\, \overline{\text{v}_{n+1}}dx\,t^2 dtd\delta\right|\\ +\,\,\frac{1}{2}\left| 
\fint_\eta^{2\eta}\! \!\int_\delta^{1/\delta}\! \!\int_{\mathbb{R}^n} 
\partial_t^2 u\, \overline{\partial_t\text{v}_{n+1}}dx\,t^2 dtd\delta\right| \,\,+\,\,|\mathcal{B}|\\\equiv |I| + |II| + 
|\mathcal{B}|,\end{multline} where we have again integrated by parts in $t$, and $\mathcal{B}$ denotes boundary terms which satisfy
\begin{multline*}|\mathcal{B}|  \leq
C\left(\sup_{r>0} \fint_r^{2r}\!\!\int_{\mathbb{R}^n}r^4|\partial_t^2 u(x,t)|^2dxdt
\right)^{1/2}\left( \sup_{t>0}\|{\bf v}(\cdot,t)\|_2\right)\\\leq C\left( \sup_{t>0}\|u(\cdot,t)\|_2\right)
\left( \sup_{t>0}\|{\bf v}(\cdot,t)\|_2\right),\end{multline*}
by a double application of Caccioppoli's inequality.
Moreover,
$$|II| \leq C \,\||t^2 \partial_t^2u\|| \, \,\||t\partial_t {\bf v}\|| \leq
C \,\||t \partial_t u\|| \, \,\||t\partial_t {\bf v}\||,$$ where we have used Caccioppoli in Whitney boxes to
bound the first factor.  Turning to the main term, we have that
\begin{equation}\label{eq3.4}|I| = C\left|\fint_\eta^{2\eta}
\! \!\int_{\delta/2}^{1/(2\delta)}\! \!\int_{\mathbb{R}^n} 
\partial_t^3 u(x,2t)\, \overline{\text{v}_{n+1}(x,2t)}\,dx\,t^2 dtd\delta \right|,\end{equation}
where we have made the change of variable $t \to 2t$.
For $t$ momentarily fixed, set $$g_t(x) \equiv \partial_t u(x,t).$$
By Theorem \ref{t2.6} (i.e., the result of \cite{AAAHK}), we have that
$$\tilde{u}_t(\cdot,s) \equiv \mathcal{D}_s \left(-\frac{1}{2}I + K \right)^{-1} g_t$$
is the unique solution of  (D2) with data $g_t$.  Hence, by $t$-independence,
$$\tilde{u}_t(\cdot,s) = \partial_tu(\cdot,t+s).$$ Setting $s=t$, we therefore obtain that
$$(D_{n+1}^3 u)(\cdot,2t) =\left( \partial_t^2\mathcal{D}_t\right)\left(-\frac{1}{2}I + K \right)^{-1} g_t.$$
We observe that by \eqref{2.1a}, \eqref{eq1.5} and \eqref{eq2.2a},
\begin{equation*} ad\!j\left(\partial_t^2\mathcal{D}_{t}\right)=
\partial_{\nu^*}\partial_t^2S^*_{-t}.\end{equation*}
Consequently,  \eqref{eq3.4} becomes
$$|I| = C\left|\fint_\eta^{2\eta}\! \!\int_{\delta/2}^{1/(2\delta)}\! \!\int_{\mathbb{R}^n} 
\left(-\frac{1}{2}I + K\right)^{-1}\partial_t u(\cdot,t)\, 
\overline{\partial_{\nu^*}D_{n+1}^2S^*_{-t}\text{v}_{n+1}(\cdot,2t)}\,dx\,t^2 dtd\delta \right|,$$
so by Cauchy-Schwarz and Theorem \ref{t2.6}, it suffices to prove
\begin{equation}\label{eq3.5}\||t^2\nabla D_{n+1}^2 S^*_{-t}
{\bf v}(\cdot, 2t)\||\leq C\left( \||t\nabla {\bf v}\|| + \|N_*{\bf v}\|_2\right).\end{equation}
The left hand side of \eqref{eq3.5} equals
\begin{multline}\label{eq3.6}\left(\sum_{k=-\infty}^\infty \int_{2^k}^{2^{k+1}}
\!\!\int_{\mathbb{R}^n} |t^2\nabla D_{n+1}^2 S_{-t}^* {\bf v}(\cdot,2t)
|^2 \frac{dx\, dt}{t}\right)^{1/2}\\\leq\,\,\left(\sum_{k=-\infty}^\infty \int_{2^k}^{2^{k+1}}
\!\!\int_{\mathbb{R}^n} |t^2\nabla D_{n+1}^2 S_{-t}^*\left( {\bf v}(\cdot,2t)
-{\bf v}(\cdot,2t_k)\right)|^2 \frac{dx\, dt}{t}\right)^{1/2}\\
+\,\,\left(\sum_{k=-\infty}^\infty \int_{2^k}^{2^{k+1}}
\!\!\int_{\mathbb{R}^n} |t^2
\nabla D_{n+1}^2 S_{-t}^*{\bf v}(\cdot,2t_k)|^2 \frac{dx\, dt}{t}\right)^{1/2}\equiv
\,\,III+IV\end{multline}
where $t_k = 2^{k-1}$.
We consider term $IV$ first.  Dividing $\mathbb{R}^n$ into cubes of side length
$2^k$, and using Caccioppoli's inequality, we deduce that
\begin{multline*}IV\, \leq \,C \left(\sum_{k=-\infty}^\infty \int_{2^{k-1}}^{2^{k+2}}
\!\!\int_{\mathbb{R}^n} |t D_{n+1}^2 S_{-t}^*{\bf v}(\cdot,2t_k)|^2 \frac{dx\, dt}{t}\right)^{1/2}\\\leq\, 
C\left(\sum_{k=-\infty}^\infty \int_{2^{k-1}}^{2^{k+2}}
\!\!\int_{\mathbb{R}^n} |t D_{n+1}^2 S_{-t}^*\left( {\bf v}(\cdot,2t_k)
-{\bf v}(\cdot,2t)\right)|^2 \frac{dx\, dt}{t}\right)^{1/2}\\
+\,\,\||\left(tD_{n+1}^2S_{-t}^*1\right)\left(P_t {\bf v}(\cdot,2t)\right)\||\,\,+
\,\,||R_t {\bf v}(\cdot,2t)\||\,\,\equiv\,\, IV_1 + IV_2 + IV_3,
\end{multline*}
where $$R_t \equiv tD_{n+1}^2S_{-t}^*- \left(tD_{n+1}^2S_{-t}^*1\right)P_t,$$
and $P_t$ is a nice approximate identity with a smooth, compactly supported kernel.
By Lemma \ref{l2.15}, 
$$IV_3 \leq \||t\nabla {\bf v}\||.$$
By \eqref{eq2.squarefunction}, Lemma \ref{l2.9a}, and a well known argument of Fefferman and Stein
\cite{FS}, we have that
$|t \partial_t^2 S^*_{-t} 1|^2\frac{dx dt}{t}$ is a Carleson measure, whence
$$IV_2 \leq C \|N_* {\bf v}\|_2.$$
By Lemma \ref{l2.9a}, the operator $f\to tD_{n+1}^2 S_{-t}^* f$ 
is bounded on $L^2(\mathbb{R}^n)$, so that
\begin{multline*}IV_1 \leq C\left(\sum_{k=-\infty}^\infty \int_{2^{k-1}}^{2^{k+2}}
\!\!\int_{\mathbb{R}^n}\left|\frac{1}{\sqrt{t}}\int_{2t_k}^{2t}\partial_s{\bf v}(x,s)ds\right|^2 dx\, dt\right)^{1/2}
\\\leq \, C \left(\sum_{k=-\infty}^\infty 
\iint_{\mathbb{R}^{n+1}}\left|\fint_{2t_k}^{2t}
1_{\{2^{k}\leq s \leq 2^{k+3}\}}\sqrt{s}\partial_s{\bf v}(x,s)ds\right|^2 dx\, dt\right)^{1/2}\leq C
\||t \partial_t {\bf v}\||,
\end{multline*}
where in the last step we have used the boundedness of the Hardy-Littlewood Maximal function.

Finally, we consider term $III$ in \eqref{eq3.6}.  By Lemma \ref{l2.11a}, 
$III$ may be handled like $IV_1$ above.  We omit the details. \end{proof}

\section{The domain of the generator of the Poisson semigroup \label{s4}}
In this section we generalize to our setting a result of \cite{D} concerning the domain of the 
generator of the Poisson semigroup.
We continue to suppose that the hypotheses of Theorem \ref{t1.3} hold.
By Theorem \ref{t2.6}, if $\epsilon \leq \epsilon_0$ is sufficiently small,
then the Dirichlet problem (D2) has a unique solution.  Consequently, the solution operator 
$f\to\mathcal{P}(t) f \equiv u(\cdot,t),$
where $u$ solves (D2) with data $f$, satisfies
\begin{equation}\label{eq4.1}\sup_{t>0}\|\mathcal{P}(t)\|_{2\to2} \leq C,\quad
\lim_{t\to0}\|\mathcal{P}(t) f - f\|_2 =0,
\end{equation}
and 
\begin{equation}\label{eq4.2}\mathcal{P}(t+s) =\mathcal{P}(t)\mathcal{P}(s), 
\end{equation}
where the last identity uses also $t$-independence of the coefficients. 
Standard semigroup theory therefore implies that the semigroup
$\{\mathcal{P}(t)\}$ has a densely defined infinitesimal generator 
on $L^2(\mathbb{R}^n)$, which we denote by $\mathcal{A}$.
We will show that the domain $D(\mathcal{A})$ of this generator is the Sobolev space
$L^2_1 \equiv L^2 \cap\dot{L}^2_1$.   More precisely, we have the following
\begin{theorem}\label{t4.3}Suppose that $L$ is as above.  Then $D(\mathcal{A}) = L^2_1,$
and
$$\|\mathcal{A}f\|_2 \approx \|\nabla_xf\|_2.$$
\end{theorem}
We remark that this last theorem can be viewed as an extension of the Kato square root problem
(\cite{CMcM},\cite{HMc}, \cite{AHLT},\cite{HLMc} and \cite{AHLMcT}) to the case that the coefficient
matrix $A$ is a full $(n+1) \times (n+1)$ matrix.  Indeed, the Kato problem corresponds to the case
that the coefficient matrix has the special ``block" structure
\begin{equation}
\left[\begin{array}{c|c}
 & 0\\ B & \vdots\\
 & 0\\
\hline 0\cdots0 & 1\end{array}\right]\label{eq4.4}\end{equation}
 where $B=B(x)$ is a $n\times n$ matrix.  In the latter case
 the generator of the Poisson semigroup is $$-\sqrt{-\dv_x B \nabla_x},$$
 and the conclusion of Theorem \ref{t4.3} is the (now established) Kato conjecture.
 
 We also note that by Theorem \ref{t2.6}, we have the representation
\begin{equation}\label{eq4.represent}
\mathcal{P}(t) = \mathcal{D}_t\left(-\frac{1}{2}I + K\right)^{-1}.\end{equation}
 
 In order to prove the theorem we shall require the following result from
 \cite{AAAHK}.
 \begin{theorem}\cite{AAAHK}\label{t4.5}
 Suppose that L satisfies the hypotheses of Theorem \ref{t1.3}.  
 There exists a small constant $\epsilon_0= \epsilon_0(n,\lambda,\Lambda)$ such 
 that if $\epsilon$ in \eqref{eq1.small} satisfies $\epsilon \leq \epsilon_0$,
 then the single layer potential satisfies
 \begin{equation}\label{eq4.6} \sup_{t \in \mathbb{R}}\|\nabla S_t \|_{2\to 2} \leq C,
 \end{equation}
 and $S_0\equiv S_t|_{t=0} : L^2(\mathbb{R}^n) \to \dot{L}^2_1(\mathbb{R}^n)$ is a bijection.
 Moreover, there is a unique solution to the Regularity problem
 \begin{equation}
\tag{R2}\begin{cases} Lu=0\text{ in }\mathbb{R}_{+}^{n+1}\\ u(\cdot,t)\to 
f\in\dot{L}_{1}^{2}(\mathbb{R}^{n}) \, n.t.\\
\widetilde{N}_{\ast}(\nabla u)\in L^{2}(\mathbb{R}^{n}),\end{cases}\label{R2}\end{equation}
which has the representation \begin{equation} \label{eq4.rep}
u(\cdot,t) \equiv S_t \left(S_0^{-1} f\right),\end{equation}
and $\partial_t u(\cdot,t)$ converges $n.t.$  and in $L^2(\mathbb{R}^n)$ as $t\to 0$.
Finally, \begin{equation}\label{eq4.7}\left(\nabla S_{t}\right)|_{t=\pm s} f
\to  \mp\frac{1}{2}\cdot\frac{f(x)}{A_{n+1,n+1}(x)}e_{n+1}+\mathcal{T}f 
\end{equation}
weakly in $L^2(\mathbb{R}^n)$, where $\mathcal{T} : L^2(\mathbb{R}^n) 
\to L^2(\mathbb{R}^n,\mathbb{C}^{n+1}) $ 
 (see \cite{AAAHK}, Lemma 4.18 for a precise definition of $\mathcal{T}$). \end{theorem}
 \begin{proof}[Proof of Theorem \ref{t4.3}]
 The deep results underlying Theorem \ref{t4.3} are 
Theorems \ref{t2.6} and  \ref{t4.5}, and we shall deduce
the first as a straightforward corollary of the latter two.
We observe that if $u$ solves (R2) with data $f \in L^2_1, $ then 
$$ \lim_{t\to 0} \partial_t u (\cdot,t) = \lim_{t\to 0} \partial_t\,\mathcal{P}(t) f \equiv \mathcal{A} f. $$
Thus, by \eqref{eq4.6}, \eqref{eq4.rep} and the bijectivity of $S_0$,
$$\|\mathcal{A}f\|_2 \leq C \|\nabla_xf\|_2.$$
  
 The proof of the opposite inequality is only a bit harder, and we sketch the details briefly here.
 We modify slightly the strategy of Verchota \cite{V}.
 By the well known Rellich identity (see, e.g., \cite{K}), and the case $\epsilon = 0$ of 
 Theorem \ref{t4.5}, for $A_0$ real and symmetric
 we have that
 \begin{equation}\label{eq4.9}\|\nabla_x u_0(\cdot,t)\|_2 \approx 
 \|\partial_t u_0 (\cdot,t)\|_2,\end{equation}
 uniformly in $t\geq 0$,
 when $u_0(\cdot,t) \equiv S^0_t f,$ and $S^0_t$ is the single layer potential
 asociated to $L_0 \equiv -\dv A_0 \nabla.$   By Theorem \ref{t4.5} and analytic perturbation theory,
 \begin{equation}\label{eq4.lip}\|\left(\nabla S_t^0 - \nabla S_t\right) f\|\leq C \|A_0-A\|_\infty \,
 \|f\|_2.\end{equation}  The latter estimate, combined with
\eqref{eq4.9} yields, uniformly in $t\geq 0$,
 \begin{equation*}C^{-1}\|\partial_t S_t f\|_2-C\epsilon_0\|f\|_2\leq
 \|\nabla_x S_t f\|_2 \leq C\|\partial_t S_t f\|_2 + C\epsilon_0\|f\|_2.\end{equation*}
Since the tangential derivatives $\nabla_x S_t f$ do not jump across the boundary,
the latter bound, plus its analogue for the lower half space, and \eqref{eq4.7} imply
 \begin{equation*}\left\|\frac{1}{2}(A_{n+1,n+1})^{-1}f+\mathcal{T}_{n+1}f \right\|_2\leq
 \left\|\frac{1}{2}(A_{n+1,n+1})^{-1}f-\mathcal{T}_{n+1}f \right\|_2 + C\epsilon_0\|f\|_2,\end{equation*}
where $\mathcal{T}_{n+1} \equiv \mathcal{T}\cdot e_{n+1}$.
Thus , by the accretivity of $A_{n+1,n+1}$ we have
\begin{multline*}\|f\|_2 \leq \left\|\frac{1}{2}(A_{n+1,n+1})^{-1}f+\mathcal{T}_{n+1}f \right\|_2 +
 \left\|\frac{1}{2}(A_{n+1,n+1})^{-1}f-\mathcal{T}_{n+1}f \right\|_2\\
 \leq C \left\|\frac{1}{2}(A_{n+1,n+1})^{-1}f-\mathcal{T}_{n+1}f \right\|_2 + C\epsilon_0\|f\|_2.
 \end{multline*}
 For $\epsilon_0$ small enough, we may first hide the small term, and then obtain invertibility 
 on $L^2$ of
 $-\frac{1}{2}(A_{n+1,n+1})^{-1} I +\mathcal{T}_{n+1}$ using  \eqref{eq4.lip} and
 the method of continuity as in \cite{AAAHK}.  Now, given $f \in D(\mathcal{A})$, we set
 $$\tilde{u}(\cdot,t) \equiv \mathcal{P}(t) f,\quad u(\cdot,t)\equiv S_t \left(-\frac{1}{2}(A_{n+1,n+1})^{-1}I +\mathcal{T}\right)^{-1} \mathcal{A}f,$$
 so that $\partial_t \tilde{u}(\cdot,t),\,\partial_t u(\cdot,t) \to \mathcal{A}f \, n.t.$ and in $L^2$ as $t \to 0$.
By uniqueness in (D2), $\partial_t \tilde{u} =\partial_t u$, hence
 $\tilde{u} - u$ depends only on $x$, and therefore, since $L(\tilde{u}-u) = 0$, and $\tilde{u}(\cdot,t)
 -u(\cdot,t)\in \dot{L}^2_1$, for each fixed $t>0$ (for $\tilde{u}$, this is a 
 consequence of the representation \eqref{eq4.represent}),  
 we deduce that $\tilde{u}-u = constant$.  Thus,
we have that $$\|\nabla_x f\|_2 \leq \sup_{t\geq 0}\|\nabla_x \tilde{u}(\cdot,t)\|_2
=\sup_{t\geq 0}\|\nabla_x u(\cdot,t)\|_2 \leq C\|\mathcal{A}f\|_2,$$
where in the last step we have used \eqref{eq4.6} and the 
bijectivity of  $-\frac{1}{2}(A_{n+1,n+1})^{-1} I +\mathcal{T}_{n+1}$.
  \end{proof}

\medskip
\noindent {\bf Acknowledgements}. The author thanks Zongwei Shen for posing the question.


\begin{thebibliography}{9999}

\bibitem{AAAHK} M. Alfonseca, P. Auscher, A. Axelsson,  S. Hofmann and S. Kim,
Analyticity of layer potentials and $L^{2}$ solvability of 
boundary value problems for divergence form elliptic
equations with complex $L^{\infty}$ coefficients, preprint
http://www.math.missouri.edu/ $_{\widetilde{}}$ hofmann/.


\bibitem{A} P. Auscher, Regularity theorems and heat kernel for elliptic operators. {\it J. London Math. Soc.} (2) {\bf 54}  (1996),  no. 2, 284--296.

\bibitem{AHLMcT} P. Auscher, S. Hofmann, M. Lacey, A.
McIntosh, and P. Tchamitchian, The solution of the Kato Square
Root Problem for Second Order Elliptic operators on $\mathbb{R}^n$, {\it Annals
of Math.} {\bf 156} (2002), 633--654.

\bibitem{AHLT}  P. Auscher, 
S. Hofmann, J.L. Lewis, P. Tchamitchian,
Extrapolation of Carleson 
measures and the analyticity of 
 Kato's square root operators, 
{\it Acta Math.} {\bf 187}  (2001),  no. 2, 161--190.

\bibitem{CMcM}
R.~Coifman, A.~M$^{\rm c}$Intosh, and Y.~Meyer.
\newblock {L'int\'egrale de Cauchy d\'efinit un op\'erateur born\'e sur
  $L^2$ pour les courbes lipschitziennes}.
\newblock {\em Ann. Math.} {\bf 116} 361--387, 1982.

\bibitem{D} B. Dahlberg,  Poisson semigroups and singular integrals, Proc. Amer. Math. Soc.
{\bf 97} (1986), 41-48.

\bibitem{DeG} E. De Giorgi, Sulla differenziabilit\`a e l'analiticit\`a delle estremali degli integrali multipli regolari. (Italian)  {\it Mem. Accad. Sci. Torino. Cl. Sci. Fis. Mat. Nat.} (3) {\bf 3}  (1957) 25--43.

\bibitem{FS} C.  Fefferman, and E. M. Stein,  $H\sp{p}$ spaces of several variables, {\it Acta Math.}  {\bf 129}  (1972), no. 3-4, 137--193.

\bibitem{HK} S. Hofmann and S. Kim, The Green function estimates for strongly elliptic systems of second order, to appear, Manuscripta Mathematica.

\bibitem{HLMc}
S.~Hofmann, M. Lacey and A.~M$^{\rm c}$Intosh.
\newblock {The solution of the Kato problem
for divergence form elliptic operators with Gaussian heat kernel bounds}.
\newblock {\it Annals of Math.} {\bf 156} (2002), pp 623-631.

\bibitem{HMc}  S. Hofmann and A. McIntosh, The solution of the Kato problem in two dimensions,
Proceedings of the Conference on Harmonic Analysis and 
PDE held in El Escorial, Spain in July 2000, {\it Publ. Mat.} Vol. extra, 2002
pp. 143-160.

\bibitem{K}  C. Kenig, {\it Harmonic analysis techniques for second order elliptic boundary value problems}, CBMS Regional Conference Series in Mathematics, 83. Published for the Conference Board of the Mathematical Sciences, Washington, DC; by the American Mathematical Society, Providence, RI, 1994

\bibitem{KP} C. Kenig and J. Pipher,  The Dirichlet problem for elliptic equations with drift terms,
{\it Publ. Mat.} {\bf 45} (2001), no.1, 199-217.

\bibitem{MMT} D. Mitrea, M. Mitrea and M. Taylor, Layer potentials, the Hodge Laplacian, and global boundary problems in nonsmooth Riemannian manifolds.  {\it Mem. Amer. Math. Soc.} {\bf 150}  (2001),  no. 713.

\bibitem{N} J. Nash, Continuity of solutions of parabolic and elliptic equations, {\it Amer. J. Math.} {\bf 80}  (1958) 931--954.


\bibitem{V} G. Verchota, Layer potentials and regularity for the Dirichlet problem for Laplace's equation in Lipschitz domains. {\it J. Funct. Anal.} {\bf 59}  (1984),  no. 3, 572--611.


\end{thebibliography}
\end{document}